\documentclass[11pt,reqno]{amsart}
\usepackage{amssymb,amsmath}\usepackage{subfigure}\newcommand{\be}{\begin{eqnarray}}
\newcommand{\ee}{\end{eqnarray}}
\newcommand{\beno}{\begin{eqnarray*}}
\newcommand{\eeno}{\end{eqnarray*}}
\newcommand{\eps}{{\mbox{$\epsilon$}}}\newcommand{\e}{{\varepsilon}}\newcommand{\R}{{\mathbb R}}\newcommand{\Z}{{\mathbb Z}}\newcommand{\LL}{\mathcal L}\newcommand{\Pk}{{\mathcal P}}

\newcommand{\om}{\omega}
\newcommand{\supp}{\operatorname{supp}}\newtheorem{theorem}{Theorem}\newtheorem{lemma}[theorem]{Lemma}\newtheorem{corollary}[theorem]{Corollary}\newtheorem{prop}[theorem]{Proposition}\newtheorem{conj}[theorem]{Conjecture}\theoremstyle{definition}\newtheorem{defi}[theorem]{Definition}\newtheorem{example}[theorem]{Example}\theoremstyle{remark}\newtheorem{remark}[theorem]{Remark}\numberwithin{equation}{section}\input epsf.sty
\begin{document}\thispagestyle{empty}
\newcommand{\diag}{\operatorname{diag}}
%
%
%
%
%
%
%
%
%
%
%
%
\newcommand{\G}{{\mathcal G}}
\newcommand{\C}{{\mathbb C}}
\newcommand{\Pa}{{P_{1,\theta}(y)}}
\newcommand{\Pb}{{P_{2,\theta}(y)}}
\newcommand{\Pshp}{{P_{1,t}^\sharp(x)}}
\newcommand{\Pflt}{{P_{1,t}^\flat(x)}}
\newcommand{\OTL}{{\Omega(\theta,\ell)}}
\newcommand{\rsz}{{R(\theta^*)}}
\newcommand{\phitil}{{\tilde{\varphi}_t}}
\newcommand{\lambdatil}{{\tilde{\lambda}}}
\newcommand{\al}{\alpha}
\newcommand{\ftpI}{\langle (g^2 + \tau^2 f^2)^{\frac{p}{2}} \rangle_{I}}
\newcommand{\fI}{\langle f \rangle_{I}}
\newcommand{\fpI}{\langle |f|^p \rangle_{I}}
\newcommand{\gI}{\langle g \rangle_{I}}
\newcommand{\mtp}{|(f,h_J)| = |(g,h_J)|}
\newcommand{\fJp}{\langle f \rangle_{J^+}}
\newcommand{\fJm}{\langle f \rangle_{J^-}}
\newcommand{\gJp}{\langle f \rangle_{J^+}}
\newcommand{\gJm}{\langle f \rangle_{J^-}}
\newcommand{\fIpm}{\langle f \rangle_{J^{\pm}}}
\newcommand{\gIpm}{\langle g \rangle_{J^{\pm}}}
\newcommand{\fIp}{\langle f \rangle_{I^{+}}}
\newcommand{\fIm}{\langle f \rangle_{I^{-}}}
\newcommand{\gIp}{\langle g \rangle_{I^{+}}}
\newcommand{\gIm}{\langle g \rangle_{I^{-}}}
\newcommand{\yone}{y_1}
\newcommand{\ytwo}{y_2}
\newcommand{\yonep}{y_1^+}
\newcommand{\ytwop}{y_2^+}
\newcommand{\yonem}{y_1^-}
\newcommand{\ytwom}{y_2^-}
\newcommand{\ythr}{y_3}

\newcommand{\xone}{x_1}
\newcommand{\xtwo}{x_2}
\newcommand{\xonep}{x_1^+}
\newcommand{\xtwop}{x_2^+}
\newcommand{\xonem}{x_1^-}
\newcommand{\xtwom}{x_2^-}
\newcommand{\xthr}{x_3}
\newcommand{\ptwo}{\frac{p}{2}}
\newcommand{\alp}{\al^+}
\newcommand{\alm}{\al^-}
\newcommand{\alpm}{\al^{\pm}}
\newcommand{\twrp}{\frac{2}{p}}
\newcommand{\sbeta}{\sqrt{\om^2 - \tau^2}}
\newcommand{\ssbeta}{\om^2 - \tau^2}
\newcommand{\sign}{\operatorname{sign}}
\newcommand{\rst}[1]{\ensuremath{{\mathbin\upharpoonright}%
\newcommand{\and}{\operatorname{and}}

\raise-.5ex\hbox{$#1$}}}

\title[]{{Laminates meet Burkholder functions}}
\author{Nicholas Boros}\address{Nicholas Boros, Dept. of Math., Michigan State University.{\tt borosnic@msu.edu}}
\author{L\'aszl\'o Sz\'ekelyhidi Jr.}\address{L\'aszl\'o Sz\'ekelyhidi Jr., Universit\"at Leipzig.  {\tt szekelyhidi@math.uni-leipzig.de}}
\author{Alexander Volberg}\address{Alexander Volberg, Dept. of  Math., Michigan State University. 
{\tt volberg@math.msu.edu}}

\begin{abstract}
Let $R_1$ and $R_2$ be the planar Riesz transforms.  We compute the $L^p-$operator norm of a quadratic perturbation of $R_1^2 - R_2^2$ as 
\beno
\left\|{\left(\begin{array}{c}
R_1^2-R_2^2\,,\,
\tau I\end{array}\right)}\right\|_{L^p(\C, \C) \to L^p(\C, \C^2)} = ((p^*-1)^2 + \tau^2)^{\frac 12}\,,
\eeno
for $1<p<2$ and $\tau^2 \leq p^*-1$ or $2 \leq p < \infty$ and $\tau \in \R.$  To obtain the lower bound estimate of, what we are calling a quadratic perturbation of $R_1^2 - R_2^2,$ we discuss a new approach of constructing laminates (a special type of probability measure on matrices) to approximate the Riesz transforms.
\end{abstract}
\maketitle

\section{Introduction}
Determining the exact $L^p-$operator norm of a singular integral is a difficult task to accomplish in general.  Classically, the Hilbert transform's operator norm was determined by Pichorides \cite{Pic}.  More recently, the real and imaginary parts of the Ahlfors--Beurling operator were determined by Nazarov,Volberg \cite{NaVo} and Geiss, Montgomery-Smith, Saksman \cite{GMS}.  Considering the full Alhfors--Beurling operator, the lower bound was determined as $p^*-1$ by Lehto \cite{Le}, or a new proof of this fact in \cite{BV}.  On the other hand the upper bound has been quite a bit more difficult.  Iwaniec conjectured in \cite{Iw} that the upper bound is $p^*-1.$  However, attempts at getting the conjectured upper bound of $p^*-1$ have been unsuccessful so far.  The works of Ba\~nuelos,Wang \cite{BaWa}, Nazarov, Volberg \cite{NaVo}, Ba\~nuelos, M\'endez-Hern\'andez \cite{BaMH}, Dragi\v{c}evi\`c, Volberg \cite{DrVo}, Ba\~nuelos, Janakiraman \cite{BaJa}, and Borichev, Janakiraman, Volberg \cite{BJV} have progressively gotten closer to $p^*-1$ as the upper bound, but no one has yet achieved it.  We note that all of these upper bound estimates crucially rely on Burkholder's estimates \cite{Bu} of the martingale transform.  Burkholder's estimates of the martingale transform even play a crucial part in determining the sharp lower bounds of the real and imaginary parts of the Alhfors-Beurling operator as seen in Geiss, Montgomery-Smith, Saksman \cite{GMS}.  Also, in \cite{GMS} it becomes clear that if one can determine the $L^p-$operator norm of some perturbation of the martingale transform then one can use it to determine the $L^p-$operator norm of the same perturbation of the real or imaginary part of the Alhfors-Beurling operator.  In \cite{BJV1}, \cite{BJV2} and \cite{BJV3} the $L^p-$operator norm was computed for a ``quadratic perturbation" of the martingale transform using the Bellman function technique, which is similar to how Burkholder originally did, for $\tau = 0,$ in \cite{Bu}.   By ``quadratic perturbation", we are referring to the quantity $(Y^2+\tau^2X^2)^{\frac 12},$ where $\tau \in \R$ is small, $X$ is a martingale and $Y$ is the corresponding martingale transform.  

We claim that our operator \beno
\left\|{\left(\begin{array}{c}
R_1^2-R_2^2\,,\,
\tau I\end{array}\right)}\right\|_{L^p(\C, \C) \to L^p(\C, \C^2)} \,,
\eeno represents a simpler model of the difficulties encountered by treating the Ahlfors--Beurling transform. An extra interesting feature of this operator is that there it breaks the symmetry between $p\in (1,2)$ and $p\in (2,\infty)$. Another interesting feature is that it is {\it simpler} to treat than a seemingly simpler perturbation $R_1^2-R_2^2 +\tau I$, whose norm in $L^p$ is horrendously difficult to find. The last assertion deserves a small elaboration. Burkholder found the norm in $L^p$ of the martingale transform, where the family of transforming multipliers $\epsilon_I$ run over $[-1,1]$. Nobody knows  how (and it seems very difficult) to find the norm of  the martingale transform, where the family of transforming multipliers $\epsilon_I$ run over $[-0.9,1.1]$. That would be the norm in $L^p$ of $R_1^2-R_2^2 +0.1\,I= 1.1 R_1^2-0.9 R_2^2$. However, if one writes the perturbation in the form we did this above, the problem immediately becomes much more treatable, and we manage to find the precise formula for the norm for a wide range of $p$'s and $\tau$'s. 

The method of Bourgain \cite{Bo}, for the Hilbert transform, which was later generalized for a large class of Fourier multiplier operators by Geiss, Montgomery-Smith, Saksman \cite{GMS}, is to discretize the operator and generalize it to a higher dimensional setting.  This operator in the higher dimensional setting will turn out to have the same operator norm and it naturally connects with discrete martingales, if done in a careful and clever way.  At the end, one has the operator norm of the singular integral bounded below by the operator norm of the martingale transform, which Burkholder found in \cite{Bu}.  This approach can be used for estimating $\left\|{\left(\begin{array}{c}
R_1^2-R_2^2\,,\,
\tau I\end{array}\right)}\right\|_{L^p \to L^p}$ from below as well, see \cite{BV}. However, we will present an entirely different approach to the problem.  

Rather than working with estimates on the martingale transform, we only need to consider the ``Burkholder-type" functions that were used to find those sharp estimates on the martingale transform.  More specifically, we analyze the behavior of the ``Burkholder-type" functions, $U$ and $v$ found in \cite{BJV1}, \cite{BJV2} and \cite{BJV3} (and reiterated in Definition \ref{definitionUv} and Theorem \ref{leastbiconcave}), associated with determining the $L^p-$operator norm of the quadratic perturbation of the martingale transform.  Using the fact that $U$ is the least bi-concave majorant of $v$ (in the appropriately chosen coordinates), in addition to some of the ways in which the two functions interact will allow us to construct an appropriate sequence of laminates, which approximate the push forward of the $2-$dimensional Lebesgue measure by the Hessian of a smooth function with compact support.  Once the appropriate sequence of laminates is constructed, we are finished since the norm of the Reisz transforms can be approximated by certain fractions of the partial derivatives of smooth functions.  The beauty of this method is that it quickly gets us the sharp lower bound constant with very easy calculations.  This lower bound argument is discussed in Section \ref{section2}.

The use of laminates for obtaining lower bounds on $L^p$-estimates has been first recognized by D.~Faraco. In \cite{DF} he introduced the so-called staircase laminates. These have also been used in refined versions of convex integration in \cite{AFS} in order to construct 
special quasiregular mappings with extremal integrability properties. Staircase laminates also proved useful in several other problems, see for instance \cite{CFM}.

In the current note we will use a continuous, rather than a discrete laminate. More importantly, in contrast to the techniques used in \cite{DF,AFS,CFM} we will construct the laminate indirectly using duality. The advantage is essentially computational, we very quickly obtain the sharp lower bounds. Indeed, the lower bound follows from the following inequality 
$$
f(1,1)\leq \frac{1}{(1-k)}\int_1^\infty [f(kt,t)+f(t,kt)]t^{-p_k}\frac{dt}{t},
$$
where $k\in (-1,1)$ and $p_k=\frac{2}{1-k}$, and valid for all biconvex functions $f:\R^2\to\R$ with $f(x)=o(|x|^{p_k})$ as $|x|\to\infty$ (see Section \ref{s:laminate}).

The ``Burkholder--type" functions $U$ and $v$ also play a crucial role in obtaining the sharp upper bound estimate as well.  With the ``Burkholder--type" functions we are able to extend sharp estimates of $(Y^2+\tau^2X^2)^{\frac 12},$ obtained in \cite{BJV1}, \cite{BJV2} and \cite{BJV3}, from the discrete martingale setting to the continuous martingale setting.  The use of ``heat martingales", as in \cite{BaMH} and \cite{BaJa}, will allow us to connect the Riesz transforms to the continuous martingales estimate, without picking up any additional constants.  This upper bound argument is presented in Section \ref{s:upperbound}. 

\section{Lower Bound Estimate}\label{section2}

\subsection{\bf Laminates and gradients}\hfill

We denote by $\R^{m\times n}$ the space of real $m\times n$ matrices and by $\R^{n\times n}_{sym}$ the set of real $n\times n$ symmetric matrices.

\begin{defi}
We say that a function $f:\R^{m\times n} \to \R$ is \textit{rank-one convex}, if $t\mapsto f(A+tB)$ is convex for all $A,B \in \R^{m\times n}$ with $\textrm{rank }B= 1$.
\end{defi}

Let $\Pk(\R^{m\times n})$ denote the set of all compactly supported probability measures on $\R^{m \times n}$.
For $\nu \in \Pk$ we denote by $\overline{\nu} = \int_{\R^{m\times n}}{X d\nu(X)}$ the center of mass or \textit{barycenter} of $\nu.$

\begin{defi}\label{d:laminate}
A measure $\nu \in \Pk$ is called a \textit{laminate}, denoted $\nu\in\mathcal{L}$, if 
\begin{equation}\label{e:deflam}
f(\overline{\nu}) \leq \int_{\R^{m\times n}}{f d\nu}
\end{equation} 
for all rank-one convex functions $f$. The set of laminates with barycenter $0$ is denoted by $\mathcal{L}_0(\R^{m\times n})$. 
\end{defi}

Laminates play an important role in several landmark applications of convex integration for producing unexpected counterexamples, see for instance \cite{MS99,KMS,AFS,SzCI,CFM}. For our purposes the case of $2\times 2$ symmetric matrices is of relevance, therefore in the following we restrict attention to this case. The key point is that laminates can be viewed as probability measures recording the gradient distribution of maps, see Corollary \ref{c:approxlaminate} below. This is by now a very standard technique.
For the convenience of the reader we provide the main steps of the argument. Detailed proofs of these statements can be found for example in \cite{MS99,Kirchheim,SzCI}. 

\begin{defi}\label{d:prelam}
Given a set $U\subset\R^{2\times 2}$ we call $\mathcal{PL}(U)$ the set of \emph{prelaminates} generated in $U$. This is the smallest
class of probability measures on $\R^{2\times 2}$ which
\begin{itemize}
\item contains all measures of the form $\lambda \delta_A+(1-\lambda)\delta_B$ with $\lambda\in [0,1]$ and $\textrm{rank}(A-B)=1$;
\item is closed under splitting in the following sense: if $\lambda\delta_A+(1-\lambda)\tilde\nu$ belongs to $\mathcal{PL}(U)$ for some $\tilde\nu\in\Pk(\R^{2\times 2})$ and $\mu$ also belongs to $\mathcal{PL}(U)$ with $\overline{\mu}=A$, then also $\lambda\mu+(1-\lambda)\tilde\nu$ belongs to $\mathcal{PL}(U)$.
\end{itemize}
\end{defi} 

The order of a prelaminate denotes the number of splittings required to obtain the measure from a Dirac measure.

\begin{example}\label{ex:2lam}
The measure
$$
\frac{1}{4}\delta_{\diag(1,1)}+\frac{1}{4}\delta_{\diag(-1,1)}+\frac{1}{2}\delta_{\diag(0,-1)},
$$
where 
$$
\diag(x,y):=\begin{pmatrix}x&0\\0&y\end{pmatrix},
$$
is a second order prelaminate with barycenter $0$.
\end{example}

It is clear from the definition that $\mathcal{PL}(U)$ consists of atomic measures. Also, from a repeated application of Jensen's inequality it follows that $\mathcal{PL}\subset\mathcal{L}$. The following two results are standard (see \cite{AFS,Kirchheim,MS99,SzCI}). 

\begin{lemma}\label{l:prelaminate}
Let $\nu=\sum_{i=1}^N\lambda_i\delta_{A_i}\in\mathcal{PL}(\R^{2\times 2}_{sym})$ with $\overline{\nu}=0$. Moreover, let
$0<r<\tfrac{1}{2}\min|A_i-A_j|$ and $\delta>0$. For any bounded domain $\Omega\subset\R^2$ there exists $u\in W^{2,\infty}_0(\Omega)$ such that $\|u\|_{C^1}<\delta$ and for all $i=1\dots N$
$$
\bigl|\{x\in\Omega:\,|D^2u(x)-A_i|<r\}\bigr|=\lambda_i|\Omega|.
$$
\end{lemma}

\begin{lemma}\label{l:laminate}
Let $K\subset\R^{2\times 2}_{sym}$ be a compact convex set and $\nu\in\mathcal{L}(\R^{2\times 2}_{sym})$ with $\supp\nu\subset K$. For any relatively open set $U\subset\R^{2\times 2}_{sym}$ with $K\subset\subset U$ there exists a sequence $\nu_j\in \mathcal{PL}(U)$ of prelaminates with $\overline{\nu}_j=\overline{\nu}$ and $\nu_j\overset{*}{\rightharpoonup}\nu$.
\end{lemma}

By combining Lemmas \ref{l:prelaminate} and \ref{l:laminate} and using a simple mollification, we obtain the following statement, linking laminates supported on symmetric matrices with second derivatives of functions.

\begin{corollary}\label{c:approxlaminate}
Let $\nu\in\mathcal{L}_0(\R^{2\times 2}_{sym})$. Then there exists a sequence $u_j\in C_c^{\infty}(B_1(0))$ with uniformly bounded second derivatives, such that
$$
\int_{B_1(0)} \phi(D^2u_j(x))\,dx\,\to\,\int_{\R^{2\times 2}_{sym}}\phi\,d\nu
$$
for all continuous $\phi:\R^{2\times 2}_{sym}\to\R$.

\end{corollary}

\subsection{\bf Laminates and lower bounds}\hfill

Let $\tau \in \R$ be fixed.  Our goal is to find
\be\label{goal}
\sup_{\varphi\in\mathcal{S}(\R^2)}\frac{\left\|\bigl( (R_1^2\varphi - R_2^2\varphi)^2+\tau^2 (R_1^2\varphi + R_2^2\varphi)^2\bigr)^{1/2} \right\|_p}{\|R_1^2\varphi + R_2^2\varphi\|_p},
\ee
for the planar Riesz transforms $R_1$ and $R_2$, where $\mathcal{S}(\R^2)$ is the Schwartz class. We can rework the Riesz transforms acting on $\varphi$ into the 
second derivative of a function $u \in \mathcal{S}(\R^2)$ in the following way.
\beno
R_i^2 \varphi = \left(-\frac{\xi_i^2}{|\xi|^2}\widehat{\varphi} \right)^{\vee} = \partial_i^2u,
\eeno
where  `` $\widehat{}$ " denotes the Fourier transform, ``$\vee{}$" denotes the inverse Fourier transform and $-\Delta u=\varphi$.  So (\ref{goal}) is equivalent to 
\be\label{Fourierreduction}
\sup_{u\in\mathcal{S}(\R^2)}\frac{\int{|(\partial_{11}^2u - \partial_{22}^2u)^2+\tau^2(\partial_{11}^2u + \partial_{22}^2u)^2|^{\frac p2}}}{\int{|\partial_{11}^2u + \partial_{22}^2u|^p}}.
\ee
Let $A_{ij}$ denote, as usual, the $ij$-entry of a matrix $A$ and put
\begin{equation}\label{e:phi}
\begin{split}
\phi_1(A) &= |(A_{11} - A_{22})^2+\tau^2(A_{11} + A_{22})^2|^{\frac p2}\, ,\\
\phi_2(A) &= |A_{11} + A_{22}|^p\,.
\end{split}
\end{equation}
Using a standard cut-off argument we can write replace $\mathcal{S}(\R^2)$ with $C_c^{\infty}(\R^2)$ and write (\ref{Fourierreduction}) as 
\be\label{pushforward}
\sup_{u\in C^{\infty}_c(\R^2)}\frac{\int{\phi_1(D^2u) dx}}{\int{\phi_2(D^2u) dx}}.
\ee
From Corollary \ref{c:approxlaminate} we deduce that
\be\label{reductiontouselaminatethm}
\sup_{u\in C^{\infty}_c(\R^2)}\frac{\int{\phi_1(D^2u)\, dx}}{\int{\phi_2(D^2u)\, dx}}\quad \geq\, \sup_{\nu\in\mathcal{L}_0(\R^{2\times 2}_{sym})} \frac{\int{\phi_1\, d\nu}}{\int{\phi_2\, d\nu}}. 
\ee
We remark in passing that, whether for arbitrary continuous functions $\phi_1,\phi_2$ one has equality above, is directly related to the celebrated conjecture of Morrey concerning rank-one and quasiconvexity in $\R^{2\times 2}_{sym}$, see for instance \cite{MuLecturenote}.
 
Our goal is therefore to prove the following
\begin{theorem}\label{t:laminate}
For any $1<p<\infty$ and $\tau\in\R$ 
there exists a sequence $\nu_N\in\mathcal{L}_0(\R^{2\times 2}_{sym})$ such that
$$
\frac{\int{\phi_1 d\nu_N}}{\int{\phi_2 d\nu_N}} \,\to\, ((p^*-1)^2+\tau^2)^{\tfrac{p}{2}}.
$$
\end{theorem}

\bigskip

\subsection{\bf Proof of Theorem \ref{t:laminate}}\label{s:laminate}\hfill 

A function $f(x,y)$ of two variables is said to be biconvex if the functions $x\mapsto f(x,y)$ and $y\mapsto f(x,y)$ are convex for all $x,y$. We start with the following inequality for biconvex functions in the plane.

\begin{lemma}\label{l:biconvexity}
Let $k\in (-1,1)$ and $N>1$.
For every $f\in C(\R^2)$ biconvex we have
\begin{equation}\label{e:main}
f(1,1)\leq \frac{1}{1-k}\int_1^N\bigl(f(kt,t)+f(t,kt)\bigr)t^{-\tfrac{2}{1-k}}\frac{dt}{t}\,+f(N,N)N^{-\tfrac{2}{1-k}}.
\end{equation}
\end{lemma}

\begin{proof}
By a standard regularization argument it suffices to show the inequality for 
$f\in C^1(\R^2)$ biconvex.
The biconvexity implies the following elementary inequalities:
\begin{eqnarray}
f(t,t)&\leq& \lambda^{\eps} f(t,t+\eps)+(1-\lambda^{\eps})f(t,kt),\label{e:cvx1}\\
f(t,t+\eps)&\leq&	\mu^{\eps} f(t+\eps,t+\eps) + (1-\mu^{\eps})f(k(t+\eps),t+\eps),\label{e:cvx2}
\end{eqnarray} 
where
\begin{eqnarray*}
\lambda^{\eps}&=&1-\frac{\eps}{t(1-k)+\eps}\\
\mu^{\eps}&=&1-\frac{\eps}{t(1-k)+\eps(1-k)}.
\end{eqnarray*}
Combining \eqref{e:cvx1} and \eqref{e:cvx2} and observing that $\lambda^\eps,\mu^\eps=1-\frac{\eps}{t(1-k)}+o(\eps)$,
we obtain
\begin{equation}\label{diffquotient}
\begin{split}
\frac{f\bigl(t+\eps,t+\eps\bigr)-f(t,t)}{\eps}&-\frac{2}{t(1-k)}f(t+\eps,t+\eps)\geq\\
 &-\frac{1}{t(1-k)}\biggl(f(k(t+\eps),t+\eps)+f(t,kt)\biggr)+o(1).
\end{split}
\end{equation}
Letting $\eps\to 0+$ this yields
$$
-\frac{\partial}{\partial t}f(t,t)+\frac{2}{t(1-k)}f(t,t)\leq \frac{1}{t(1-k)}\bigl(f(kt,t)+f(t,kt)\bigr)
$$
Multiplying both sides by $t^{-\tfrac{2}{1-k}}$ and integrating, we obtain \eqref{e:main} as required.
\end{proof}
The method of obtaining continuous laminates by integrating a differential inequality as above is due to Kirchheim, and appeared first in the context of separate convexity in $\R^3$ in \cite{KMS}. 

Next, for $1<p<\infty$ let 
$$
k=1-\frac{2}{p},
$$
so that $p=\frac{2}{1-k}$. We need to differentiate between the cases $1<p\leq 2$ and $2<p<\infty$.

\bigskip

\noindent{\bf The case $1<p\leq 2$.}

Let $\mu_{N}\in\Pk(\R^{2\times 2})$ be defined by the RHS of \eqref{e:main}, more precisely
$$
\int f\,d\mu_{N}:=\frac{1}{1-k}\int_1^N\bigl[f\bigl(\diag(kt,t)\bigr)+f\bigl(\diag(t,kt)\bigr)\bigr]t^{-p}\frac{dt}{t}\,+\frac{f\bigl(\diag(N,N)\bigr)}{N^{p}}
$$
for $f\in C(\R^{2\times 2})$. Then $\mu_{N}$ is a probability with barycenter $\overline{\mu}_{N}=\diag(1,1)$. Moreover, observe that if $f$ is rank-one convex, then $(x,y)\mapsto f(\diag(x,y))$ is biconvex. Therefore, using Lemma \ref{l:biconvexity} we see that $\mu_N$ is a laminate. Then, combining with the measure from Example \ref{ex:2lam} (c.f.~splitting procedure from Definition \ref{d:prelam}) we conclude that the measure
$$
\nu_{N}:=\frac{1}{4}\mu_{N}+\frac{1}{4}\delta_{\diag(-1,1)}+\frac{1}{2}\delta_{\diag(0,-1)}
$$
is a laminate with barycenter $\overline{\nu}_{N}=0$. We claim that this sequence of laminates has the desired properties for Theorem \ref{t:laminate}. To this end we calculate
\begin{align*}
\int\phi_1\,d\mu_N&=p|(1-k)^2+\tau^2(1+k)^2|^{p/2}\log N+2^p, \\
\int\phi_2\,d\mu_N&=p(1+k)^p\log N.
\end{align*}
In particular we see that as $N\to\infty$
\begin{equation*}
\begin{split}
\frac{\int\phi_1\,d\nu_N}{\int\phi_2\,d\nu_N}\quad \to\quad &\frac{|(1-k)^2+\tau^2(1+k)^2|^{p/2}}{(1+k)^p}\\
&=\biggl[\biggl(\frac{1-k}{1+k}\biggr)^2+\tau^2\biggr]^p\\
&=\biggl[\biggl(\frac{1}{p-1}\biggr)^2+\tau^2\biggr]^p=[(p^*-1)^2+\tau^2]^p.
\end{split}
\end{equation*}

\bigskip

\noindent{\bf The case $2<p<\infty$.}

Let $\tilde\mu_{N}\in\Pk(\R^{2\times 2})$ be defined by
$$
\int f\,d\tilde\mu_{N}:=\frac{1}{1-k}\int_1^N\bigl[f\bigl(\diag(-kt,t)\bigr)+f\bigl(\diag(-t,kt)\bigr)\bigr]t^{-p}\frac{dt}{t}\,+\frac{f\bigl(\diag(-N,N)\bigr)}{N^{p}}
$$
for $f\in C(\R^{2\times 2})$. Then $\tilde\mu_{N}$ is a probability with barycenter $\overline{\tilde\mu}_{N}=\diag(-1,1)$. Moreover, as before, we see that if $f$ is rank-one convex, then $(x,y)\mapsto f(\diag(-x,y))$ is biconvex. Therefore $\tilde\mu_N$ is again a laminate, hence also
$$
\tilde\nu_{N}:=\frac{1}{4}\tilde\mu_{N}+\frac{1}{4}\delta_{\diag(1,1)}+\frac{1}{2}\delta_{\diag(0,-1)}
$$
is a laminate with barycenter $0$. Repeating the calculations above, we obtain
\begin{equation*}
\begin{split}
\frac{\int\phi_1\,d\tilde\nu_N}{\int\phi_2\,d\tilde\nu_N}\quad \underset{N\to\infty}{\longrightarrow}\quad &\frac{|(1+k)^2+\tau^2(1-k)^2|^{p/2}}{(1-k)^p}\\
&=\biggl[\biggl(\frac{1+k}{1-k}\biggr)^2+\tau^2\biggr]^p\\
&=[(p-1)^2+\tau^2]^p=[(p^*-1)^2+\tau^2]^p.
\end{split}
\end{equation*}

\section{Comparison with Burkholder functions}\hfill

Now we will discuss the ``Burkholder-type" functions introduced in \cite{BJV1}, \cite{BJV2} and \cite{BJV3}.  Let $p^*-1 := \max\left\{p-1, \frac{1}{p-1}\right\}$ and $x := (x_1, x_2)$ denote a point in $\R^2.$  We will denote the coordinates $y := (\yone, \ytwo) \in \R^2,$ as the rotation of $x$ by $\frac{\pi}{4},$ that is 
\beno
\yone = \frac{x_1+x_2}{2},\,\, \ytwo = \frac{x_1-x_2}{2}.
\eeno
\begin{defi}
We say that a function $f := f(x_1, x_2)$ is \textit{zigzag concave} if it is bi-concave in the $y-variables.$
\end{defi}
\begin{defi}\label{definitionUv}
Let $v(x_1,x_2) := v_{p, \tau}(x_1, x_2) = (\tau^2|x_1|^2+|x_2|^2)^{\frac p2}-((p^*-1)^2+\tau^2)^{\frac p2}|x_1|^p$ and $u(x_1,x_2) := u_{p, \tau}(x_1, x_2) = \alpha_p(|x_1|+|x_2|)^{p-1}[|x_2|-(p^*-1)|x_1|],$
where $\alpha_p =  p(1-\frac{1}{p*})^{p-1}\left(1+\frac{\tau^2}{(p^*-1)^2}\right)^{\frac{p-2}{2}}.$  For $1 < p <2,$ we define
\begin{displaymath}
  U(x_1, x_2) :=  U_{p,\tau}(x_1, x_2) = \left\{
     \begin{array}{lr}
       v(x_1, x_2) & : |x_2| \geq (p^*-1)|x_1|\,\,\\
       u(x_1, x_2) & : |x_2| \leq (p^*-1)|x_1|,
     \end{array}
   \right.
\end{displaymath}
and for $2<p<\infty,$
\begin{displaymath}
  U(x_1, x_2) :=  U_{p,\tau}(x_1, x_2) = \left\{
     \begin{array}{lr}
       u(x_1, x_2) & : |x_2| \geq (p^*-1)|x_1|\,\,\\
       v(x_1, x_2) & : |x_2| \leq (p^*-1)|x_1|.
     \end{array}
   \right.
\end{displaymath}
\end{defi}

\begin{defi}  
\label{cM}
Denote  $c_{\mathcal{M}} := \inf\left\{c: v_c \text{ has a zigzag concave majorant and }U\text{ is such that }U(0,0) = 0\right\}.$
\end{defi}
Now we will see the key relationship between the ``Burkholder--type" functions $U$ and $v.$
\begin{theorem}\label{leastbiconcave}
1.  $c_{\mathcal{M}} \geq ((p^*-1)^2+\tau^2)^{\frac p2},$ for all $1 < p < \infty$ and all $\tau \in \R.$\\

2.  $c_{\mathcal{M}} = ((p^*-1)^2+\tau^2)^{\frac p2},$ for $1<p<2$ and $\tau^2 \leq p^*-1$ or $2 \leq p < \infty$ and $\tau \in \R.$\\

3.  If $1 < p < 2$ and $\tau$ is sufficiently large then $c_{\mathcal{M}} > ((p^*-1)^2+\tau^2)^{\frac p2}.$
\end{theorem}

\begin{proof}
$1)$  By way of contradiction, suppose that there is such a $\widetilde{c} \in [0,((p^*-1)^2+\tau^2)^{\frac p2})$ that $v_{\widetilde{c}}$ has a zigzag concave majorant.   Then following the upper bound estimate in Section \ref{s:upperbound}, Theorem \ref{mainupperboundestimate} would have $\widetilde{c}$ as the upper bound of our quadratic perturbation.  However, this is impossible because of Theorem \ref{t:laminate}.  

$2)$  In \cite{BJV1} and \cite{BJV3} it is shown for $1<p<2$ and $\tau^2 \leq p^*-1$ or $2 \leq p < \infty$ and $\tau \in \R$ that for $c = ((p^*-1)^2+\tau^2)^{\frac p2}, v_c$ has a zigzag concave majorant.  This proves that $c_{\mathcal{M}} \leq ((p^*-1)^2+\tau^2)^{\frac p2}.$  Combining with $1)$ we get equality.

$3)$  For $1<p<2$ and $\tau \in \R$ sufficiently large, $U$ is no longer zigzag concave, while still being a majorant of  $v_c,$ with $c = ((p^*-1)^2+\tau^2)^{\frac p2}.$  We know that if $\tau$ is sufficiently large and $1<p<2,$ the least $c_0$ for which $v_{c_0}$ has a zigzag concave majorant must satisfy $c_0 > ((p^*-1)^2+\tau^2)^{\frac p2}.$  See \cite{BJV3}, Remark 27.  The condition $\tau^2 \leq p^*-1$ is a sufficient condition for $U$ to be the zigzag concave majorant, but not necessary. \qedhere
\end{proof}

\subsection{{\bf Analyzing the ``Burkholder-type" functions $U$ and $v$}} \hfill

We will use the $y-$coordinates, unless otherwise stated, from this point on.  In the $y-$coordinates, 
\beno
\widetilde{U}(y_1, y_2) := U(y_1 -y_2, y_1 + y_2) = U(x_1, x_2), 
\eeno
and it takes the following form.  For $2 \leq p < \infty,$
\beno
  \widetilde{U}(y) = \left\{
     \begin{array}{lr}
       \widetilde{u}(y),\,\, \frac{p-2}{p}\yone \leq \ytwo \leq \frac{p}{p-2}\yone \\
       \widetilde{v}(y),\,\, \text{otherwise},
     \end{array}
   \right.
\eeno
and for $1 < p < 2,$ 
\be\label{Uforpless2}
   \widetilde{U}(y) = \left\{
     \begin{array}{lr}
       \widetilde{v}(y),\,\, \frac{2-p}{p}\yone \leq \ytwo \leq \frac{p}{2-p}\yone \\
       \widetilde{u}(y),\,\, \text{otherwise},
     \end{array}
   \right.
\ee
where \beno
\widetilde{u}(y_1, y_2) := u(y_1 -y_2, y_1 + y_2) = u(x_1, x_2), \widetilde{v}(y_1, y_2) := v(y_1 -y_2, y_1 + y_2) = v(x_1, x_2).  
\eeno

We will fix $2 \leq p < \infty,$ as the dual range of $p$ values is handled similarly.  Denote 
\be \label{definedconstants}
k:= \frac{p}{p-2},
\ee
then $p = \frac{2k}{k-1}$ and $p -1 = \frac{k+1}{k-1}.$  Also denote
\be \label{Lk}
L_k := \{\ytwo = k\yone\} \,\, \text{and} \,\, L_{\frac 1k} := \left\{\ytwo = \frac 1k\yone\right\}.
\ee
Observe that in the cone
\beno
C_1 = \{\yone \leq \ytwo \leq k \yone\},
\eeno
$\widetilde{U}$ is linear if we fix $\ytwo.$  Also, $\widetilde{U}$ is linear if we fix $\yone$ in the cone
\beno
C_2 = \{\frac 1k \yone \leq \ytwo \leq \yone\}.
\eeno
Consequently, $\widetilde{U}$ is almost linear in the ``T-shape" graph, which we will denote as $T$, with vertices 
\beno
\left\{(\frac 1k (\yone + h), \yone +h), (\yone, \yone + h), (\yone + h, \yone+h), (\yone, \frac 1k \yone)\right\}.
\eeno
The only portion where $\widetilde{U}$ is not linear is on the segment from $(\yone, \yone)$ to $(\yone, \yone+h).$   It is very small in comparison with the graph $T.$

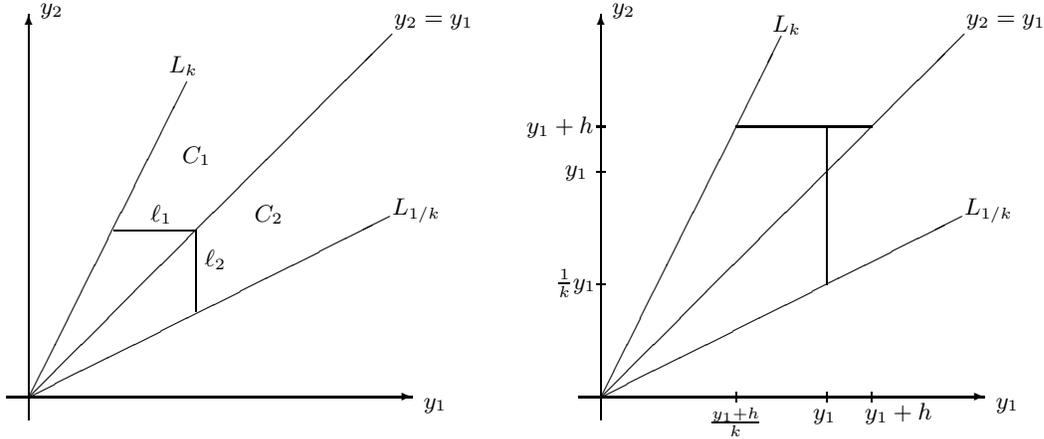
\begin{figure}[h]
  \centering
 \setlength{\unitlength}{0.03cm}
\subfigure[$\widetilde{U}$ is linear moving along lines such as $\ell_1$ and $\ell_2.$]{
\begin{picture}(200,200)(0,0)
\thinlines
\put(20,10){\vector(0,1){180}} 
\put(10,20){\vector(1,0){180}} 

\put(20,20){\line(1,1){161}} 

\put(20,20){\line(2,1){160}} 
\put(20,20){\line(1,2){70}} 

\put(58,94){\line(1,0){36}} 
\put(94,94){\line(0,-1){36}} 

\put(182,185){\footnotesize $y_2=y_1$}
\put(181,101){\footnotesize $L_{1/k}$}
\put(120,98){\footnotesize $C_2$}
\put(98,78){\footnotesize $\ell_2$}
\put(82,164){\footnotesize $L_k$}
\put(74,98){\footnotesize $\ell_1$}
\put(88,124){\footnotesize $C_1$}
\put(25,190){\footnotesize $y_2$}
\put(195,15){\footnotesize $y_1$}
\end{picture}}                
\subfigure[The ``T-shaped" graph $T.$]
{\setlength{\unitlength}{0.03cm}
\begin{picture}(245,200)(-45,0)
\thinlines
\put(20,10){\vector(0,1){180}} 
\put(10,20){\vector(1,0){180}} 

\put(20,20){\line(1,1){161}} 

\put(20,20){\line(2,1){160}} 
\put(20,20){\line(1,2){80}} 

\put(140,140){\line(-1,0){60}} 
\put(120,140){\line(0,-1){70}} 

\put(80,22){\line(0,-1){4}} 
\put(68,6){\footnotesize  $\frac{y_1+h}{k}$}
\put(120,22){\line(0,-1){4}} 
\put(114,10){\footnotesize $y_1$}
\put(140,22){\line(0,-1){4}} 
\put(137,10){\footnotesize $y_1+h$}
\put(18,70){\line(1,0){4}} 
\put(0,68){\footnotesize $\frac 1k y_1$}
\put(18,120){\line(1,0){4}} 
\put(4,118){\footnotesize $y_1$}
\put(18,140){\line(1,0){4}} 
\put(-13,137){\footnotesize $y_1+h$}

\put(182,185){\footnotesize $y_2=y_1$}

\put(181,101){\footnotesize $L_{1/k}$}
\put(96,182){\footnotesize $L_k$}

\put(25,190){\footnotesize $y_2$}
\put(195,15){\footnotesize $y_1$}
\end{picture}}
\caption{Splitting between $u$ and $\widetilde{v}$ in $y_1 \times y_2-$plane.}
\label{T-shapedregion}
\end{figure}

By Theorem \ref{leastbiconcave}, $\widetilde{U} \geq \widetilde{v}$ in $\R^2.$  But, on $L_k$ and $L_{\frac 1k}, \widetilde{U} = \widetilde{v}.$  Also, observe that $\widetilde{U}(0,0) = \widetilde{v}(0,0)$ and $\widetilde{v} \geq 0$ on $C_1$ and $C_2$ (this is easy to see in $x-$coordinates).  We will summarize these important facts, so that we can later refer to them.

\begin{prop} \label{Uproperties}
(1)  $\widetilde{v} \geq 0$ on $C_1$ and $C_2$\\
(2)  $\widetilde{v}(0,0) = \widetilde{U}(0,0)$\\
(3)  $\widetilde{U} = \widetilde{v}$ on $L_k$ and $L_{\frac 1k}.$\\
(4)  $\widetilde{U}$ is nearly linear on $T.$
\end{prop}

\subsection{Why the laminate sequence $\nu_N$ worked in Theorem 8}\hfill

Let $\phi_1(y_1,y_2):= (|y_1-y_2|^2+\tau^2|y_1+y_2|^2)^{\frac{p}2}$, $\phi_2(y_1,y_2):= |y_1+y_2|^p$.

\begin{defi}
 Let $c_{\mathcal{L}} := \sup\left\{\frac{\int\phi_1d\nu}{\int\phi_2d\nu}:  \nu\in\mathcal{L}_0\right\}.$
\end{defi}

\begin{theorem}
\label{th:c}
$c_{\mathcal{M}} \geq c_{\mathcal{L}}.$
\end{theorem}

\begin{proof}
By the definition of $c_{\mathcal{L}},$ there exists a laminate $\nu\in\mathcal{L}$ with barycenter $0$, such that $\frac{\int\phi_1 d\nu}{\int\phi_2 d\nu} > c_{\mathcal{L}} - \varepsilon.$  This is equivalent to 
\be\label{t:lowerboundcontradiction}
\int[\phi_1 - (c_{\mathcal{L}}-\varepsilon)\phi_2]d\nu > 0.
\ee
We will now show that $\phi_1 - (c_{\mathcal{L}}-\varepsilon)\phi_2$ does not have a  biconcave majorant. By changing the variables back to $x_1,x_2$ we see that this means that $v_{c_{\mathcal{L}}-\varepsilon}(x_1,x_2)$ would not have a zigzag concave majorant, thus proving that $c_{\mathcal{L}}-\varepsilon\le c_{\mathcal{M}}$.  By way of contradiction, suppose that $b := c_\LL - \e,$ and that $\phi_1 - b\phi_2 \leq U,$ which is biconcave.  Then by (\ref{t:lowerboundcontradiction}), 
\beno
0 < \int(\phi_1 - b \phi_2)d\nu \leq \int Ud\nu \leq U(\overline{\nu}) = U(0, 0) = 0.
\eeno
This gives a contradiction and we are finished with the proof.\qedhere
\end{proof}

\begin{defi}
We denote $(p, \tau) \in \mathcal{T},$ if $1<p<2$ and $\tau^2 \leq p^*-1$ or $2 \leq p < \infty$ and $\tau \in \R.$ 
\end{defi}

Let us compose $v_c(x_1, x_2)$ with the change of variables
\beno
x_1 &=& y_1+y_2\\
x_2 &=& y_1 - y_2.
\eeno
We get the function called
\beno
\widetilde{v}_c(y_1, y_2) := (|y_1 -y_2|^2+\tau^2|y_1+y_2|^2)^{\frac p2} - c|y_1+y_2|^p.
\eeno
Recall that similarly $\widetilde{U}(y_1, y_2) := U(y_1 + y_2, y_1-y_2) = U(x_1, x_2).$

Let us introduce the following notation:
\begin{defi}
$c_{\mathcal{B}}:= ((p^*-1)^2+\tau^2)^{\frac{p}2}\,.$
\end{defi}
Here $\mathcal{B}$ stands for Burkholder. 
Recall Definition \ref{cM} and 
let us also denote 
\begin{defi}
 $c_{\mathcal{N}}:=\left\|\begin{pmatrix}R_1^2 - R_2^2, & \tau I  \end{pmatrix}\right\|_{L^p(\C, \C) \to L^p(\C, \C^2)}$.
\end{defi}
\begin{remark}
For $(p, \tau) \in \mathcal{T}, c_\LL = c_{\mathcal{M}}=c_{\mathcal{B}}$.
\end{remark}
This follows immediately from Theorems \ref{t:laminate},  \ref{leastbiconcave} and \ref{th:c}.

Moreover, we saw that for all $p, \tau$
$$
c_{\mathcal{L}}\le c_{\mathcal{N}}\,.
$$
This is Corollary \ref{c:approxlaminate}   essentially.

In Section \ref{s:upperbound} we are proving that
$$
c_{\mathcal{N}} \le c_{\mathcal{M}}\,.
$$

We wish to discuss the set $\Omega$ of  pairs $(p, \tau)$, for which $c_{\mathcal{L}}=c_{\mathcal{M}}$. This set of pairs contains $\mathcal{T}$ introduced above, but we do not know exactly the whole $\Omega$. By what we have just said if $(p,\tau)\in\Omega$ then the norm of our operator is $c_{\mathcal{L}}$, and we also know that for such pairs the sharp estimate from above for the norm is obtainable by means of finding the least $c$ for which $v_c$ has a zigzag concave majorant.

Now we want to see what kind of restriction the equality $c_{\mathcal{L}}=c_{\mathcal{M}}$ imposes {\it a priori}.

\begin{conj}
If $p,\tau$ are such that $c_\LL = c_{\mathcal{M}}$ then there exists a sequence of laminates $\{\nu_N\}$ with barycenter $0,$ such that $\int{\widetilde{U}_{c_{\mathcal{M}}}d\nu_N}$ increases to $0.$
\end{conj}

\begin{remark}
In fact, this is exactly what happens on $\mathcal{T}$. Namely, if $\nu_N = \frac 14 \mu_N + \frac 14 \delta_{(-1,1)}+\frac 12 \delta_{(0, -1)}$ and $(p, \tau) \in \mathcal{T}$ then 
\beno-\mathcal{O}(1) \leq \int\widetilde{U}_{c_{\mathcal{M}}}d\nu_N \leq 0.
\eeno
\end{remark}

If $(p, \tau)$ is not in $\mathcal{T},$ but $c_\LL = c_{\mathcal M}$ then we get something interesting.  By the definition of $c_\LL,$ there exists some $\nu_N$ with barycenter $0,$ and such that $\frac{\int \phi_1 d\nu_N}{\int \phi_2 d\nu_N} \geq c_{\mathcal M} - \frac 1N.$  Then we get 
\beno
- \frac 1N \leq \frac{\int (\phi_1 - c_{\mathcal M}\phi_2) d\nu_N}{\int \phi_2 d\nu_N} \leq \frac{\int \widetilde{U}_{c_{\mathcal M}}d\nu_N}{\int \phi_2 d\nu_N} \leq 0.
\eeno
Therefore, 
\be\label{littleo}
\left|\int \widetilde{U}_{c_{\mathcal M}} d\nu_N\right| = o\left(\int\phi_2 d\nu_N\right).
\ee
For (\ref{littleo}) it would be sufficient to have 
\beno
\left|\int \widetilde{U}_{c_{\mathcal M}} d\nu_N\right| = \mathcal{O}\left(1\right).
\eeno
Remembering that $\int{\widetilde{U}_{c_{\mathcal M}}d\nu_N} \leq 0$ we can write this as
\beno
-c \leq \int{\widetilde{U}_{c_{\mathcal M}}d\nu_N} \leq 0.
\eeno
This was exactly the case for $(p, \tau) \in \mathcal T$ with $\nu_N =  \frac 14 \mu_N + \frac 14 \delta_{(-1,1)}+\frac 12 \delta_{(0, -1)}.$  The reason for that was because
\be\label{reasonforlittleo}
\overline{\mu}_N = (1, 1) \text{ and }\int{\widetilde{U}_{c_{\mathcal M}}d\mu_N} = \widetilde{U}_{c_{\mathcal M}}(1,1)
\ee
For any $w$ biconcave and $\mu$ with barycenter $(1,1)$ we have $\int w d\mu \leq w(1, 1).$ But in \eqref{reasonforlittleo} we have the case when the equality is attained. To understand better the case when the equality can be attained when integrating a biconcave function against a laminate, let us consider first a simpler question when equality is attained in integrating the usual concave function against a usual probability measure.

  If $w$ is concave, then $\int w d\mu \leq w(1, 1)$ is true for any probability measure $\mu.$  There are only two ways to get equality (i.e. $\int w d\mu = w(1, 1)$):  1) if $\mu$ is a delta measure at $(1, 1)$ or 2) if $w$ is linear on the convex hull of the support of measure $\mu$ (degenerate concave).

Coming back to the attained equality in (\ref{reasonforlittleo}), for biconcave $\widetilde{U}_{c_{\mathcal M}},$ we see that (\ref{reasonforlittleo}) happened also exactly because the Burkholder function $\widetilde{U}_{c_{\mathcal M}},$ is not only biconcave on the cones $C_1 \cup C_2,$ but degenerate biconcave, meaning that $C_1 \cup C_2$ is foliated by curves on which one of the concavities degenerates into linearity.  We may conjecture that the same geometric picture happens for those $(p, \tau)$ outside of $\mathcal T,$ for which $c_\LL = c_{\mathcal M},$ but we do not know how to prove this.  

To summarize, we have the following.
\begin{itemize}

\item  For all $p$ and $\tau, c_{\mathcal M} \geq c_{\mathcal{N}}\geq c_\LL \geq ((p^*-1)^2+\tau^2)^{\frac p2} =: c_{\mathcal{B}}(p, \tau).$ All four constants coincide at least for $(p,\tau)\in \mathcal T$.

\item  For all $p \in (1, 2),$ there exists a $\tau_0$ such that for all $|\tau| > \tau_0, c_{\mathcal M} > c_{\mathcal{B}}(p, \tau)$ (by \cite{BJV2} Remark 27).

\item  $(p, \tau)$ such that $c_\LL = c_{\mathcal M}$ holds for all $(p, \tau) \in \mathcal T,$ but may also be true outside of $\mathcal T.$

\item  By a modification of \cite{BJV2},  Remark 27, one can prove that for all $p \in (1,2),$ there exits $\tau_0$ such that, for all $|\tau| \geq \tau_0, c_{\mathcal{N}} > c_{\mathcal{B}}(p, \tau).$ 

\item We of course expect that always $c_{\mathcal{N}}=c_\LL$.
\end{itemize}

\section{Upper Bound Estimate}\label{s:upperbound}

\subsection{Background information and notation}
We will use similar notation, estimates and reasoning developed in \cite{BaMH} and \cite{BaJa}.  Let $B_t=(Z_t, T-t)$ denote space-time Brownian motion starting at $(0, T) \in \R_+^3 := \R^2 \times (0, \infty),$ where $Z_t$ is standard Brownian motion in the plane.  There is a pseudo-probability measure $P^{T}$ associated with the process and we will denote $\mathbb{E}^T$ as the corresponding expectation.

For $\phi \in C_c^{\infty}(\C),$ we denote $U_\phi(z,t),$ as the heat extension to the upper half space, in other words $U_\phi$ is the solution to 
\begin{displaymath}
   \left\{
     \begin{array}{lr}
       \partial_t U_\phi - \frac 12 \Delta U_\phi & ,\,\,R_+^3\\
       U_\phi = \phi,\,\, R^2.
     \end{array}
   \right.
\end{displaymath} 
By It\^o's formula we get the relation, 
\be \label{heatmartingale}
U_\phi(B_t) - U_\phi(B_0) = \int_0^t{\nabla U_\phi(B_s)\cdot dZ_s},
\ee
which is a martingale.  For a $2 \times 2$ matrix $A$ we denote
\beno
(A*U_\phi)_t := \int_0^t{A\nabla U_\phi(B_s)\cdot dZ_s}
\eeno
as a martingale transform.  Throughout Section \ref{s:upperbound} we will refer to the matrices \[
A_1 =
\left( {\begin{array}{cc}
 1 & 0  \\
 0 & 0  \\
 \end{array} } \right),\,\, A_2 =
\left( {\begin{array}{cc}
 0 & 0  \\
 0 & 1  \\
 \end{array} } \right),\,\, I =
\left( {\begin{array}{cc}
 1 & 0  \\
 0 & 1  \\
 \end{array} } \right). 
\] 
If we rewrite (\ref{heatmartingale}) in the form $X_t = X_1^t + i X_2^t = (I*\phi)_t = \int_0^t{\nabla U_\phi(B_s)\cdot dZ_s},$ then its martingale transform will be denoted as $Y_t = Y_1^t + i Y_2^t = ((A_1-A_2)*\phi)_t = \int_0^t{(A_1 - A_2)\nabla U_\phi(B_s)\cdot dZ_s}.$  The quadratic variation of $X_i$ and $Y_i$ are 
\beno
\langle X_i \rangle_t = \int_0^t{|\nabla U_{\phi_i}(B_s)|^2ds} = \int_0^t{\left|\left( {\begin{array}{cc}
 \partial_xU_{\phi_i}(B_s) \\
 -\partial_yU_{\phi_i}(B_s) \\
 \end{array} } \right) \right|^2ds} = \langle Y_i \rangle_t,\,\,for\,\, i = 1,2.
\eeno
Then, $\langle X \rangle_t = \langle X_1 \rangle_t+\langle X_2 \rangle_t = \langle Y_1 \rangle_t+\langle Y_2 \rangle_t = \langle Y \rangle_t.$

\begin{defi}
A process $H$ is called differentially subordinate to a process $K,$ if $\frac{d}{dt}\langle H \rangle_t \leq \frac{d}{dt}\langle K \rangle_t.$
\end{defi}

We have computed that $Y$ is differentially subordinate to $X.$  Note that $Y$ is the continuous version of the martingale transform (the discrete version of Burkholder's martingale transform is $\sum_{k=1}^n{d_k} \rightarrow \sum_{k=1}^n{\varepsilon_k d_k},$ where $\varepsilon_k\in\{\pm 1\}$ and $\{d_k\}_k$ is a martingale difference sequence and $n \in \Z_+$), since $Y$ is differentially subordinate to $X.$

\subsection{Extending the martingale estimate to continuous time martingales}
\begin{theorem}\label{burkholdertypethm}
Let $X$ and $Y$ be two complex-valued martingales, such that $Y$ is the martingale transform of $X$ (in other words $\frac{d}{dx}\langle X \rangle_t \leq \frac{d}{dx}\langle Y \rangle_t$) then $\|\tau^2|X|^2 +|Y|^2\|_p \leq ((p^*-1)^2+\tau^2)^{\frac 12}\|X\|_p,$ with the best possible constant for $1<p<2$ and $\tau^2 \leq p^*-1$ or $2 \leq p < \infty$ and $\tau \in \R.$
\end{theorem}

This was basically shown in \cite{BJV1}, \cite{BJV2} and \cite{BJV3}, but we will give the idea of the proof.  The proof here only requires the same modification to continuous time martingales as was done in \cite{BaJa}, for $\tau = 0.$  Let 
\beno
u(x,y) &=& p\left(1-\frac{1}{p^*}\right)^{p-1}\left(1+\frac{\tau^2}{(p^*-1)^2} \right)^{p-2}(|y|-(p^*-1)|x|)(|x|+|y|)^{p-1}\,\,and\\
v(x,y) &=& (\tau^2|x|^2+|y|^2)^{\frac p2} - ((p^*-1)^2+\tau^2)^{\frac 12}|x|^p.
\eeno
It was shown in Theorem \ref{leastbiconcave} that $v \leq u.$  The key properties of $u$ and $v,$ that will be used, are:\\
(1)  $v(x,y) \leq u(x,y)$\\
(2)  For all $x,y,h,k \in \C,$ if $|x||y| \neq 0$ then 
\beno
\langle h u_{xx}(x,y), h \rangle + 2\langle hu_{xy}(x,y), k\rangle + \langle ku_{yy}(x,y),k \rangle = -c_{p,\tau}(A+B+C),
\eeno
where $c_{p,\tau} > 0$ is a constant only depending on $\tau$ and $p$ and 
\beno
A &=& p(p-1)(|h|^2-|k|^2)(|x|+|y|)^{p-2},\,\,B = p(p-2)[|k|^2 - (y',k)^2]|y|^{-1(|x|+|y|)^{p-1}},\\
C &=& p(p-1)(p-2)[(x',h) + (y',k)]^2|x|(|x|+|y|)^{p-3},
\eeno
where $x' = x/|x|,\,y' = y/|y|.$\\
(3) $u(x,y) \leq 0$ if $|y| \leq |x|.$\\
Since $u$ here only differs from the one in \cite{BaJa} by a multiple of $\left(1+\frac{\tau^2}{(p^*-1)^2} \right)^{p-2},$ then the rest of the argument follows in an identical way which we briefly outline.

By It\^o's formula, 
\beno
u(X_t,Y_t) = u(X_0, Y_0) + \int_0^t{\langle u_x(X_s,Y_s),dX_s\rangle} + \int_0^t{\langle u_y(X_s,Y_s),dY_s\rangle} + \frac{I_t}{2},
\eeno
where $I_t$ contains the second order terms.  We can assume, without loss of generality that $|Y_0| \leq |X_0|,$ so that when we take expectation of $u(X_t,Y_t),$ we obtain $\mathbb{E}u(X_t,Y_t) \leq \mathbb{E}(I_t/2).$  Using property (2) above, in the martingale setting, one can obtain $I_t \leq -c_{p,\tau}\int_0^t{(|X_s|+|Y_s|)^{p-2}d(\langle X \rangle_s - \langle Y \rangle_s)} \leq 0,$ since $B,C \geq 0$ and using the differential subordinate assumption.  Therefore, $\mathbb{E}v(X_t,Y_t) \leq 0$ by property (1) above.

\subsection{Connecting the martingales to the Riesz transforms}\hfill

Now we choose $X_t := (I*U_\phi)_t$ and $Y_t := ((A_1 - A_2)*U_\phi)_t$ to obtain the following corollary of Theorem \ref{burkholdertypethm}.

\begin{corollary}
If $1<p<2$ and $\tau^2 \leq p^*-1$ or $2 \leq p < \infty$ and $\tau \in \R$ then 
\beno
\|\tau^2|(I*U_\phi)_t|^2+|((A_1 - A_2)*U_\phi)_t|^2\|_p \leq ((p^*-1)^2+\tau^2)^{\frac 12}\|(I*U_\phi)_t\|_p.
\eeno
\end{corollary}

\begin{prop}\label{martingaleconnection}
For all $\phi \in C_c^{\infty}$ and all $p \in (1, \infty),\,\,\lim_{T\to\infty}\|(I*U_\phi)_T\|_p \leq \|\phi\|_p.$
\end{prop}
This result was proven in \cite{BaMH}.

Now we will connect the martingales $X_t$ and $Y_t$ with the planar Riesz transforms,$R_1$ and $R_2$, in the following way.

\begin{prop} \label{connectmartingales} For all $\phi \in C_c^{\infty}(\C),$
\beno
\lim_{T\to\infty}\int_{\C}{[|\mathbb{E}^T(Y_T|B_T = z)|^2+\tau^2|\mathbb{E}^T(X_T|B_T = z)|^2]^{\frac p2}dz}\\
= \int_{\C}{[|(R_1-R_2^2)\phi|^2+\tau^2|(R_1+R_2^2)\phi|^2]^{\frac p2}dz}.
\eeno 
\end{prop}
This result follows almost immediately from the fact that, for all $\psi, \phi \in C_c^\infty(\C),$
\be \label{diffofRieszconnection}
\lim_{T\to\infty}\int_\C{\psi\mathbb{E}^T[Y_T|B_T = z]dz} &=& \int_\C{\psi(R_1^2-R_2^2)\phi dz}\,\,and\\ \label{sumofRieszconnection} \lim_{T\to\infty}\int_\C{\psi\mathbb{E}^T[X_T|B_T = z]dz} &=& \int_\C{\psi(R_1^2+R_2^2)\phi dz},
\ee
by \cite{BaMH}.  By (\ref{diffofRieszconnection}) and (\ref{sumofRieszconnection}) we obtain that for all $\psi, \phi \in C_c^\infty(\C),$
\beno
\lim_{T\to\infty}\int_{\C}{\psi[|\mathbb{E}^T(Y_T|B_T = z)|^2+\tau^2|\mathbb{E}^T(X_T|B_T = z)|^2]^{\frac 12}dz}\\
= \int_{\C}{\psi[|(R_1-R_2^2)\phi|^2+\tau^2|(R_1+R_2^2)\phi|^2]^{\frac 12}dz}.
\eeno 

\subsection{Main Result}\hfill
\begin{theorem}\label{mainupperboundestimate}
For $1<p<2$ and $\tau^2 \leq p^*-1$ or $2 \leq p < \infty$ and $\tau \in \R$ we have the following estimate,
\beno
\|[|(R_1^2-R_2^2)f|^2+\tau^2|f|^2]^{\frac 12}\|_p \leq ((p^*-1)^2+\tau^2)^{\frac 12}\|f\|_p
\eeno
\end{theorem}
Let $\mathbb{E}^{z_0,T}$ correspond to Brownian motion starting at $(z_0,T) \in \R^3_+.$  Let $\phi \geq 0$ and $\frac 1p + \frac 1q = 1.$  Then

\beno
&& \int_{\C}{(|(R_1^2-R_2^2)\phi|^2+\tau^2|\phi|^2)^{\frac 12}}\psi(z)dz \\
&=& \int_{\C}{\lim_{T \to \infty}\int_\C{(|\mathbb{E}^{(z_0,T)}(Y_T|B_T=z)|^2+\tau^2|\mathbb{E}^{(z_0, T)}(X_T|B_T =z)|^2)^{\frac 12}dz_0}\psi(z)dz}\\
&=& \lim_{T\to \infty}\int_{\C}{\int_\C{(|\mathbb{E}^{(z_0,T)}(Y_T\psi(Z_T)|B_T=z)|^2+\tau^2|\mathbb{E}^{(z_0, T)}(X_T\psi(Z_T)|B_T =z)|^2)^{\frac 12}dz}dz_0}\\
& \leq & \lim_{T \to \infty}\int_{\C}{\mathbb{E}^{(z_0,T)}\left| {\left(\begin{array}{c}
 Y_T\psi(Z_T)\\
 \tau X_T\psi(Z_T) \end{array}\right)}\right|dz_0}\\
& \leq & \left(\lim_{T \to \infty} \int_\C{\mathbb{E}^{(z_0,T)}\left| {\left(\begin{array}{c}
 Y_T\\
 \tau X_T\end{array}\right)}\right|^p dz_0} \right)^{\frac 1p}   \left(\lim_{T \to \infty} \int_\C{\mathbb{E}^{(z_0,T)}| \phi(Z_T)|^q dz_0}\right)^{\frac 1q}\\
\eeno
\beno
&=& \left(\lim_{T \to \infty} \int_\C{\mathbb{E}^{(z_0,T)}\left| {\left(\begin{array}{c}
 Y_T\\
 \tau X_T\end{array}\right)}\right|^p dz_0} \right)^{\frac 1p}   \|\psi\|_{L^q},\\
\eeno
where the last equality is by Proposition \ref{martingaleconnection}.  By linearity we have this result for any $\psi \in L^q.$  Therefore, by duality
\beno
&&\left(\int_{\C}{(|(R_1^2-R_2^2)\phi|^2+\tau^2|\phi|^2)^{\frac p2}} \right)^{\frac 1p}  \leq  \left(\lim_{T \to \infty} \int_\C{\mathbb{E}^{(z_0,T)}\left| {\left(\begin{array}{c}
 Y_T\\
 \tau X_T\end{array}\right)}\right|^p dz_0} \right)^{\frac 1p}\\
& = &(\lim_{T \to \infty}\mathbb{E}^T[(|Y_T|+\tau^2|X_T|^2)^{\frac p2}])^{\frac 1p} \leq   ((p^*-1)^2+\tau^2)^{\frac 12} \lim_{T \to \infty}(\mathbb{E}|X_T|^p)^{\frac 1p}\\
& = & ((p^*-1)^2+\tau^2)^{\frac 12}\|\phi\|_{L^p},
\eeno
where the last inequality is due to Theorem \ref{burkholdertypethm} and the last equality is by Proposition \ref{connectmartingales}.

\begin{corollary}
For $1<p<2$ and $\tau^2 \leq p^*-1$ or $2 \leq p < \infty$ and $\tau \in \R,$
\beno
\left\|{\left(\begin{array}{c}
2R_1R_2\,,\,
\tau I\end{array}\right)}\right\|_{L^p(\C, \C) \to L^p(\C, \C^2)}& = & \left\|{\left(\begin{array}{c}
R_1^2-R_2^2\,,\,
\tau I\end{array}\right)}\right\|_{L^p(\C, \C) \to L^p(\C, \C^2)}\\
& = & ((p^*-1)^2 + \tau^2)^{\frac 12}.
\eeno
\end{corollary}

Since $R_1^2-R_2^2$ and $2R_1R_2$ are just rotations of one another by $\pi/4$ then we have the equality of the two operator norms.  The lower bound was computed as $((p^*-1)^2 + \tau^2)^{\frac 12}$ in Theorem \ref{t:laminate} (or by another technique in \cite{BV}).  The upper bound was just computed as the same, giving the desired result.

\bibliographystyle{amsplain}

\end{document}